\documentclass[11pt]{article}
\usepackage{}

\usepackage{amssymb}
\usepackage{latexsym}
\usepackage{extarrows}
\usepackage[dvips]{graphicx}
\usepackage{graphicx,color}
\usepackage{hyperref}

\let\cal=\mathcal
\newtheorem{theorem}{Theorem}[section]

\newtheorem{problem}[theorem]{Problem}
\newtheorem{lemma}[theorem]{Lemma}

\newtheorem{fact}[theorem]{Fact}

\newcommand{\qed}{\hspace*{\fill} \rule{7pt}{7pt}}

\begin{document}

\title{{\bf Stabilities for non-uniform $t$-intersecting families}\thanks{This paper was firstly announced on January 15, 2024, and it was later published on Electron. J. Combin. 31 (4) (2024), \#P4.3. 
This is the final version; see \url{https://doi.org/10.37236/12706}.  E-mail addresses: ytli0921@hnu.edu.cn (Y. Li), 
wu@hunnu.edu.cn (B. Wu, corresponding author)}}

\author{
Yongtao Li$^{\dag}$,  
Biao Wu$^{\ddag}$\\[2ex]
 {\small $^{\dag}$HNP-LAMA, School of Mathematics and Statistics, Central South University} \\
{\small Changsha, Hunan, 410083, P.R. China } \\
{\small $^{\ddag}$CHP-LCOCS, School of Mathematics and Statistics, Hunan Normal University} \\
{\small Changsha, Hunan, 410081, P.R. China }}

\maketitle

\vspace{-0.7cm}
\begin{abstract}
The study of intersection problems on families of sets 
is one of the most important topics in extremal combinatorics. 
As is well-known, extremal problems involving certain intersection constraints are  equivalent to those with  certain union constraints  
by taking complement of sets. 
A family of sets is called $s$-union if
the union of any two sets in this family  has size at most $s$.
 Katona [Acta Math. Hungar. 15 (1964)] 
provided the maximum size of an $s$-union family of sets of $[n]$, 
and he also determined 
 the extremal families up to isomorphism.
Recently, Frankl [J. Combin. Theory Ser. B 122 (2017) 869--876] sharpened this result by
establishing the maximum size of an $s$-union family that is
not a subfamily of the so-called Katona family. 
In this paper, we determine the maximum size of
an $s$-union family that is neither contained in 
the Katona family nor in the Frankl family.
Moreover, we characterize all extremal
families achieving the upper bounds.
\end{abstract}



\section{Introduction}

Let $[n]=\{1,2,\ldots ,n\}$.
The power set $2^{[n]}$ consists of the $2^n$ subsets of $[n]$.
For every subset $F\subseteq [n]$,
we denote by $F^c$ the complement set of $F$ in $[n]$.
We write ${[n] \choose k}$ for the collection
of all $k$-element subsets of $[n]$,
and ${[n] \choose \le k}$
for the collections of all subsets of $[n]$
with size at most $k$.
Let $\mathcal{F}$ be a family of subsets of $[n]$.
We say that $\mathcal{F}$ is $k$-uniform if
all sets of $\mathcal{F}$ have size $k$.
If $\mathcal{F}$ is non-uniform, we usually denote
 $\mathcal{F}_i= \{F\in \mathcal{F} : |F|=i\}= \mathcal{F} \cap
{[n] \choose i}$. 
Given a set $G\subseteq [n]$ and a permutation
$\sigma \in S_n $, we denote $\sigma (G)=\{\sigma (g): g\in G\}$. 
We say that $ \mathcal{G}$ and
$ \mathcal{H}$ are isomorphic if
there is a permutation
$\sigma \in S_n $ such that
$\{\sigma (G): G\in \mathcal{G}\} =\mathcal{H}$. 
For isomorphic families  $\mathcal{G}$ and $\mathcal{H}$,  we denote $\mathcal{G}=\mathcal{H}$  whenever there are no confusions. 
For two families $\mathcal{G}$ and $ \mathcal{H}$,
we say that $\mathcal{G}$ is a subfamily of
$\mathcal{H}$, denoted by $\mathcal{G} \subseteq \mathcal{H} $,
 if there is a permutation 
$\sigma \in S_n$ such that
$\sigma (G)\in \mathcal{H}$ for every $G \in \mathcal{G}$.

\subsection{Uniform intersecting families}

A family $\mathcal{F}$  of sets is called
{\it $t$-intersecting} if $|A\cap B|\ge t$ for all $A,B\in \mathcal{F}$.
For $t=1$, we just say `intersecting' instead of `$1$-intersecting'.
A {\it full star} is  a family that consists of all the $k$-subsets of $[n]$
that contains a fixed element. 
The celebrated Erd\H{o}s--Ko--Rado theorem \cite{EKR1961} states that
when $n\ge 2k$, a full star is the unique intersecting $k$-uniform family 
attaining the maximum size.

\begin{theorem}[Erd\H{o}s--Ko--Rado \cite{EKR1961}] \label{thm-EKR}
Let $n\ge 2k$ and $\mathcal{F} \subseteq {[n] \choose k}$ be an 
intersecting family.
Then
\begin{equation*}
  |\mathcal{F}| \le {n-1 \choose k-1}.
  \end{equation*}
 When $n>2k$,
  the equality holds if and only if
  $\mathcal{F}=\{F\in \tbinom{[n]}{k}: i\in F\}$
  for some $i\in [n]$.
\end{theorem}

The Erd\H{o}s--Ko--Rado theorem
is widely regarded as a cornerstone of extremal combinatorics
and has various generalizations and applications. 
For related  problems, we recommend \cite{FT2018} for the interested readers.  
Erd\H{o}s, Ko and Rado \cite{EKR1961} also proved that
there exists an integer $n_0(k,t)$ such that if $n\ge n_0(k,t)$ and
$\mathcal{F}\subseteq {[n] \choose k}$
is $t$-intersecting, then $|\mathcal{F}|\le {n-t \choose k-t}$.
The smallest possible such $n_0(k,t)$
is $(t+1)(k-t+1)$. This was proved by Frankl \cite{Fra1976} for $t\ge 15$,
and then completely solved by Wilson \cite{Wil1984} for all $t$. 

\begin{theorem} [Exact Erd\H{o}s--Ko--Rado Theorem \cite{Fra1976,Wil1984}] \label{thm21}
Let $k> t\ge 1$ be integers and let 
$\mathcal{F} \subseteq
{[n] \choose k} $
be a $t$-intersecting family.
If ${n\ge (t +1)(k- t +1)}$,
then
\[  |\mathcal{F}|\le {n-t\choose k-t}. \]
The equality holds if and only if
$\mathcal{F}$ is isomorphic to 
$\bigl\{F\in \tbinom{[n]}{k} : [t] \subseteq F
\bigr\}$ or 
$\{F \in {[n] \choose k}: |F\cap [t+2]|\ge t+1\}$. 
For the case $n> (t+1)(k-t+1)$, the former family is the unique extremal family.  
\end{theorem}

Stability results for uniform $t$-intersecting families can be found in 
\cite{Fra1978,Fra1976,AK1996}. 
Moreover, the study of stabilities for $r$-wise $t$-intersecting families 
has risen in popularity in the past few years; see \cite{OV2021,BL2022,CLW2021,CLLW2022}.   
Note that the Erd\H{o}s--Ko--Rado theorem bounds the maximum size of $k$-uniform 
$t$-intersecting families.  In 1964, Katona \cite{Kat1964}  studied 
the problem for {\it non-uniform} $t$-intersecting families.

\begin{theorem}[Katona \cite{Kat1964}] 
\label{thm-kat-left-shift}
Let $n\ge t\ge 2$ be integers  and  let 
$\mathcal{F} \subseteq 2^{[n]}$ be a $t$-intersecting family. \\ 
{\rm (1)} If $n+t=2a$ for an integer $a\ge 1$, then 
\[ |\mathcal{F}| \le \sum_{k\ge a} {n \choose k}. \]
The equality holds if and only if 
$\mathcal{F}=\{F\subseteq [n]: |F|\ge a\}$. \\
{\rm (2)} If $n+t=2a+1$ for an integer $a\ge 1$, then 
\[ |\mathcal{F}| \le  {n-1 \choose a} + 
\sum_{k\ge a+1} {n \choose k}. \]
The equality holds if and only if 
$\mathcal{F}=\{F\subseteq [n]: |F| \ge a+1\} 
\cup {[n-1] \choose a}$. 
\end{theorem}

For the case $t=1$, it is easy to see that every (non-uniform) intersecting family 
of sets of $[n]$ has size at most $2^{n-1}$, and there are 
many extremal families attaining this bound. 
There are various extension and generalization of Katona's theorem;
see, e.g., \cite{Kle1966,MT1989,FT2013,Fra2017,Fra2017cpc} for more details.

\subsection{Non-uniform families with $s$-union property}

We say that a family $\mathcal{F}$ has the 
{\it $s$-union property}, 
or simply that $\mathcal{F}$ is {\it $s$-union} if
$|F\cup F'|\le s$ holds for every pair $F,F'\in \mathcal{F}$. 
In this paper, we mainly investigate the
 $s$-union families. 
We define $m(n,s)$ as the maximum of $|\mathcal{F}|$
over all $\mathcal{F}\subseteq 2^{[n]}$ having the $s$-union property.
Clearly, we have $m(n,0)=1$, $m(n,1)=2$, $m(n,n)=2^n$ and $m(n,n-1)=2^{n-1}$, 
where the last equality holds by considering the $2^{n-1}$ pairs
$\{F,F^c\}$ for each $F\in 2^{[n]}$ and $\mathcal{F}$ 
 contains at most one set of such pairs.

In what follows, we focus mainly on stabilities for 
{\it non-uniform} $t$-intersecting families. 
In fact, the extremal problems for $t$-intersecting families can be reduced to 
that for $s$-union families. Indeed, 
it is easy to see that $\mathcal{F}$ is $t$-intersecting if and only if the dual family 
$\mathcal{F}^c :=\{[n]\setminus F : F\in \mathcal{F}\}$ 
is $(n-t)$-union. Note that  $\mathcal{F}$ and 
$\mathcal{F}^c$ have the same size.  
For notational convenience, we shall study the extremal problems 
in the language of the $s$-union property, rather than $t$-intersecting property. 
In particular, Katona's result in Theorem \ref{thm-kat-left-shift}
can be equivalently written as the following.  

\begin{theorem}[Katona \cite{Kat1964}] \label{thm1.1}
Let $2\le s \le n-2$ be integers and let $\mathcal{F}\subseteq 2^{[n]}$
be  $s$-union. \\
{\rm (1)} If $s=2d$ for an integer $d\ge 1$, then
\[ |\mathcal{F}| \le \sum_{0 \le i \le d} {n \choose i},\]
 with equality if and only if
$\mathcal{F}= \mathcal{K}(n,2d) := {[n] \choose \le d}$. \\  
{\rm (2)} If $s=2d+1$ for an integer $d\ge 1$, then
\[ |\mathcal{F}| \le \sum_{0\le i\le d} {n \choose i}
+ {n-1 \choose d}, \]
with equality  if and only if
 $\mathcal{F} \!=\! \mathcal{K}(n,2d+1) \!:=\! {[n] \choose \le d} 
\cup \bigl\{F \!\in \! {[n] \choose d+1},y\in F \bigr\} $  
for some $y\in [n]$.
\end{theorem}

In 2017, Frankl \cite{Fra2017}  proved a stability result for Katona's theorem; 
that is, he determined the maximum size of an $s$-union family
that is not a subfamily of the Katona family $\mathcal{K}(n,s)$.

\begin{theorem}[Frankl \cite{Fra2017}] \label{thm1.2}
Let $2\le s \le n-2$ be integers and let 
$\mathcal{F}\subseteq 2^{[n]}$
be  $s$-union. \\
{\rm (1)} If $s=2d$ for an integer $d\ge 1 $
and $\mathcal{F}\nsubseteq \mathcal{K}(n,2d)$, then
\[ |\mathcal{F}| \le \sum_{0 \le i \le d} {n \choose i} -
{n-d-1 \choose d} +1.\]
Moreover, the equality holds if and only if
$\mathcal{F}$ is isomorphic to
\[ \begin{matrix}
\mathcal{H}(n,2d) :=
{[n] \choose \le d-1} \cup \{D\} \cup 
 \{H \in {[n] \choose d}: H\cap D\neq \emptyset\} 
 \end{matrix} \]
 for some set $D\in {[n] \choose d+1}$. For the case
 $s=4$, apart from $\mathcal{H}(n,4)$,
 there is one more possibility, namely,
$ \mathcal{H}^*(n,4) := {[n] \choose \le 1} \cup
\{H\in \tbinom{[n]}{2} : H\cap [2] \neq \emptyset \} \cup
\bigl\{ \{1,2,i\} : i\in [3,n]   \bigr\} .$ \\ 
{\rm (2)} If $s=2d+1$ for an integer $d\ge 1$
and $\mathcal{F} \nsubseteq \mathcal{K}(n,2d+1)$,
then
\[ |\mathcal{F}| \le \sum_{0\le i\le d} {n \choose i}
+ {n-1 \choose d} - {n-d-2 \choose d} +1 .\]
Moreover, the equality holds if and only if
 $\mathcal{F}$ is isomorphic to
\[ \begin{matrix}
\mathcal{H}(n,2d+1) := {[n] \choose \le d}
\cup \{D\}
\cup \{H \in {[n] \choose d+1}: y\in H, H\cap D\neq \emptyset\}
\end{matrix} \]
for some fixed element $y\in [n]$ and
 set $D\subseteq [n] \setminus \{y\}$ with $|D|=d+1$.
 For the case $s=5$, there is also one more possibility, namely,
$  \mathcal{T}(n,5) := {[n] \choose \le 2} \cup
\bigl\{F\in \tbinom{[n] }{ 3}:
|F\cap [3]|\ge 2 \bigr\}.$
 \end{theorem}

We remark here that
the extremal family $\mathcal{H}^*(n,4)$ is missed in \cite{Fra2017}. 
As mentioned early, 
the analogous stability results 
have been studied 
 for uniform intersecting families; see 
 \cite{BL2022,CLW2021,CLLW2022,
HK2017,KM2017,HP2022,Kup2019,OV2021} for some recent progresses.  
Motivated by Theorems \ref{thm1.1} and \ref{thm1.2}, 
one could ask the following stability problem.

\begin{problem} \label{prob1}
What is the maximum size of an $s$-union family of sets of $[n]$ 
that is neither a
subfamily of the Katona family $\mathcal{K}(n,s)$ nor
of the Frankl family $\mathcal{H}(n,s)$?  
\end{problem}

In this paper, we shall solve Problem \ref{prob1} by 
studying the stabilities 
for {\it non-uniform} families and characterizing all extremal
families achieving the maximum size. 
Our approach adopts some similar ideas from Frankl \cite{Fra2017}, where the Hilton-Milner theorem are used, while in our setting, we need to apply a further stability of Han and Kohayakawa; see Section \ref{sec2}. Moreover, we need to prove a result for cross-intersecting families. Our result refines the previous bound due to Frankl \cite{FT1992,Fra}; see Section \ref{sec3}.  

To begin with, we show the case $s\in \{2,3\}$.

\begin{itemize}

\item 
For  $s=2$, recall that
$\mathcal{K}(n,2)=\{F\subseteq [n] : |F|\le 1\}$
and $\mathcal{H}(n,2)=\bigl\{ \emptyset, \{d_1\}, \{d_2\}$, $\{d_1,d_2\} \bigr\}$
for some fixed set $\{d_1,d_2\}\subseteq [n]$.  
If $\mathcal{F}$ is $2$-union, then $\mathcal{F}$ must be
 a subfamily of $\mathcal{K}(n,2)$ or $\mathcal{H}(n,2)$.

\item 
For $s=3$, we know that
$\mathcal{K}(n,3)=\bigl\{ \emptyset, \{1\}, \ldots ,\{n\}\bigr\} \cup
\{F \in {[n] \choose 2} : y\in F\}$ for some fixed $y\in [n]$,
and $\mathcal{H}(n,3)=\bigl\{ \emptyset, \{1\},  \ldots ,\{n\}\bigr\} \cup
\bigl\{ \{d_1,d_2\} , \{y,d_1\},\{y,d_2\} \bigr\} $ for some
$\{d_1,d_2\} \subseteq [n] \setminus \{y\}$. 
If $\mathcal{F}$ is $3$-union and
$\mathcal{F} $ is neither a subfamily of
$\mathcal{K}(n,3)$ nor of $\mathcal{H}(n,3)$, then
$\mathcal{F}$ contains a $3$-element set, 
whence $|\mathcal{F}|$  is maximized for 
$ \bigl\{  \emptyset, \{a,b,c\},\{a\},\{b\},\{c\} \bigr\}$
for some $\{a,b,c\} \subseteq [n]$.
\end{itemize}

To state our results, 
we define the extremal families formally.  
Let $d\ge 3$ and $n\ge 2d$ be integers, and 
$D_1,D_2 \subseteq [n]$ with
$|D_1|=|D_2|=d+1$ and $|D_1\cap D_2| =d$.  
We define two families as below.
 \[ \begin{matrix}
\mathcal{W}(n,2d) := {[n] \choose \le d-1 } \cup \{D_1,D_2\} \cup 
 \bigl\{ H \in {[n] \choose d}: H\cap D_1
 \neq \emptyset ~\text{and}~ H\cap D_2 \neq \emptyset \bigr\}. 
 \end{matrix} \]
 Moreover, let $\mathcal{J}_2(n,d+1)$  be a family introduced in
Subsection \ref{sec2.1}. We define
\[ \mathcal{W}(n,2d+1) := \tbinom{[n] }{ \le d}\cup \mathcal{J}_2(n,d+1), \]
In addition, for the case $s=6$,
we need to define two exceptional families:  
\[ \mathcal{W}^*(n,6) := \tbinom{[n]}{\le 2} \cup
\{F\in \tbinom{[n]}{ 3} :
F\cap [3] \neq \emptyset \} \cup
\bigl\{ H\in \tbinom{[n]}{4}: [3]\subseteq H\bigr\} \]
and 
\[ 
\mathcal{W}^{**}(n,6) := \tbinom{[n]}{\le 2} \cup
\{F\in \tbinom{[n]}{3} :
F\cap [2] \neq \emptyset \} \cup
\bigl\{ H\in \tbinom{[n]}{4}: [2]\subseteq H  \bigr\}.  
 \]
For the case $s=7$, we also have two exceptional families as below 
\[ \mathcal{W}^*(n,7) := \tbinom{[n]}{ \le 3}
\cup
\{ H \in \tbinom{[n] }{ 4} : \{2,3\} \subseteq H\}
\cup \{H \in \tbinom{[n] }{ 4}: 1\in H, H\cap \{2,3\} \neq \emptyset\}  \] 
and 
\[ 
\mathcal{W}^{**}(n,7):= \tbinom{[n]}{ \le 3} 
\cup
\{ H \in \tbinom{[n] }{ 4} : \{2,3,4\} \subseteq H\}
\cup \{H \in \tbinom{[n] }{ 4}:
1\in H, H\cap \{2,3,4\} \neq \emptyset\} . \] 

In the following, we shall characterize the $s$-union families for
every  $s\in  [4, n-2]$.

\begin{theorem}[Main result] \label{thmmain}
Let $4\le s \le n-2$ be integers and
let $\mathcal{F}\subseteq 2^{[n]}$ be an $s$-union family
such that
 $\mathcal{F}\nsubseteq \mathcal{K}(n,s)$
and $\mathcal{F}\nsubseteq \mathcal{H}(n,s)$.
Then the following statement holds.\\
{\rm (1)} If $s=2d$ for an integer $d\ge 2$ and further $\mathcal{F} \nsubseteq
\mathcal{H}^*(n,4)$ for the case $s=4$, then
\[ |\mathcal{F}|\le \sum_{0\le i \le d}{n\choose i}-{n-d-1 \choose d}- {n-d-2 \choose d-1}+2. \]
For $s=6$,
the equality holds  if and only if
$\mathcal{F}$ is isomorphic to
 $\mathcal{W}(n,6)$ or $\mathcal{W}^*(n,6)$ or $\mathcal{W}^{**}(n,6)$; 
for other $s=2d$,
equality holds  if and only if
$\mathcal{F}$ is isomorphic to $\mathcal{W}(n,2d)$.\\
{\rm (2)} If $s=2d+1$ for an integer $d\ge 2$
and further $\mathcal{F} \nsubseteq
\mathcal{T}(n,5)$ for the case $s=5$, then
\[ |\mathcal{F}|\le \sum_{0\le i \le d}{n\choose i}+{n-1\choose d}-{n-d-2 \choose d}- {n-d-3 \choose d-1}+2.\]
 For  $s=7$, the equality holds  if and only if
$\mathcal{F}$ is isomorphic to
 $\mathcal{W}(n,7)$ or $\mathcal{W}^*(n,7)$ or $\mathcal{W}^{**}(n,7)$;
for  other $s\!=\!2d+1$,
equality holds  if and only if
$\mathcal{F}$ is isomorphic to
 $\mathcal{W}(n,2d+1)$.
\end{theorem}

\noindent 
{\bf Organization.} 
The paper is organized as follows. In Section \ref{sec2},
we review some basic preliminaries, including the stabilities of intersecting families,  Katona's inequality, and the shifting operation. 
In Section \ref{sec3},
we shall give a sharp upper bound on the maximum
of the sum of  sizes of two cross-intersecting
families in which one of them is $2$-intersecting
(Theorem \ref{cross-intersectingtheo}), which plays a significant role in the proof of our main result.
In Section \ref{sec4}, we shall give the 
proof of Theorem \ref{thmmain}. Some ideas of our proof  are motivated by Frankl's papers  \cite{Fra2017,Fra}.   
In the last section, we conclude with some possible problems.

\section{Preliminaries}
\label{sec2}

\subsection{Stabilities for  intersecting families}
\label{sec2.1}

We say that an intersecting family $\mathcal{F}$ is
{\it trivial} if all its members  share a common element. 
In other words,
an intersecting family is called trivial if it is a subfamily of a full star.
Erd\H{o}s, Ko and Rado \cite{EKR1961} asked for the maximum size of
a nontrivial intersecting family of $k$-element subsets of $[n]$.
In 1967, Hilton and Milner \cite{HM1967} answered this question
by proving the following result. 
We denote by
$\mathcal{EKR}(n,k)$ the family of all
$k$-element subsets of $[n]$
containing a fixed element. 

\begin{theorem}[Hilton--Milner \cite{HM1967}] \label{thmhm}
Let $k\ge 2$ and $n\ge 2k$ be integers and $\mathcal{F}\subseteq {[n] \choose k}$ be
an intersecting  family.
 If  $\mathcal{F}\nsubseteq \mathcal{EKR}(n,k)$,  then
\begin{equation*}
 |\mathcal{F}| \le {n-1 \choose k-1} -{n-k-1 \choose k-1} +1.
 \end{equation*}
 Moreover, for $n>2k$, the equality holds if and only if
$\mathcal{F}$ is isomorphic to
\[ \begin{matrix}
\mathcal{HM}(n,k):=
\Bigl\{G\in {[n] \choose k}:
1\in G, G\cap [2,k+1] \neq \emptyset \Bigr\}
\cup \Bigl\{ [2,k+1]  \Bigr\}.
\end{matrix} \]
 or  in the case of $k=3$,
there is one more possibility, namely
\[ \begin{matrix}
\mathcal{T}(n,3):=\left\{F\in {[n] \choose 3}:
|F\cap [3]|\ge 2 \right\}.
\end{matrix} \]
\end{theorem}


In 2017, Han and Kohayakawa \cite{HK2017}
 determined the maximum size
of a non-trivial intersecting uniform family that
is not a subfamily of the Hilton--Milner family.
To proceed, we introduce some notation and define families  $\mathcal{J}_i(n,k)$
and $\mathcal{G}_i(n,k)$.
Let $k\ge 3$, $i\le k-1$ and $n>2k$ be positive integers.
For any $(k-1)$-element set $E\subseteq [n]$,
 any $(i+1)$-element set $J\subseteq [n] \setminus E$, $x_0\in J$ and $J_i=J\setminus \{x_0\}$,  we define
\[  \begin{matrix}
\mathcal{J}_i(n,k) :=
\left\{G \in {[n] \choose k}: x_0\in G,
G\cap (E\cup \{j\}) \neq \emptyset ~\text{for each $j\in J_i$} \right\}
\cup \Bigl\{ E\cup \{j\}:j\in J_i \Bigr\}.
\end{matrix}  \]
We next define the family $\mathcal{G}_i(n,k)$.
Suppose now that $i\in [2,k]$,
 $x_0 \in [n]$ and $E\subseteq [n] \setminus \{x_0\}$ is an
$i$-element set.
We define the
$k$-uniform family $\mathcal{G}_i(n,k)$
as 
\[  \begin{matrix}
\mathcal{G}_i(n,k) :=\left\{G \in {[n] \choose k}: E\subseteq G \right\}
\cup \left\{G \in {[n] \choose k}: x_0 \in G,
G \cap E \neq \emptyset \right\}.
\end{matrix}  \]

The  result of Han and Kohayakawa \cite{HK2017}
can be stated as below.

\begin{theorem}[Han--Kohayakawa \cite{HK2017}] \label{thmhk}
Let $k\ge 3$, $n>2k$ and $\mathcal{H}$ be an intersecting $k$-uniform
family of subsets of $[n]$. If $\mathcal{H} \nsubseteq \mathcal{EKR}(n,k)$ and $\mathcal{H} \nsubseteq \mathcal{HM}(n,k)$,
and if $k=3$, $\mathcal{H} \nsubseteq \mathcal{G}_2(n,3)$, then
\[  |\mathcal{H}|\le {n-1 \choose k-1}
- {n-k-1 \choose k-1}- {n-k-2 \choose k-2} +2. \]
For $k=4$, the equality holds if and only if  $\mathcal{H}=\mathcal{J}_2(n,4)$, $\mathcal{G}_2(n,4)$ or
$\mathcal{G}_3(n,4)$; for every other $k$, equality holds if and only if
$\mathcal{H}=\mathcal{J}_2(n,k)$.
\end{theorem}

For more stability results on uniform intersecting families, 
we refer the interested readers to the recent papers  \cite{KM2017,HP2022,Kup2019}.  
The following lemma was provided Katona \cite{Kat1964}; see \cite{Fra2017} 
for a detailed proof. 

\begin{lemma}[See \cite{Kat1964,Fra2017}] 
\label{lem-Kat-ineq}
If $\mathcal{F}\subseteq 2^{[n]}$ has $s$-union property, then
for every $i\in [0,s/2]$, 
\begin{equation} \label{eqkat}
|\mathcal{F}_i| + |\mathcal{F}_{s+1-i}| \le {n \choose i}.
\end{equation}
Moreover, for $n\ge s+2$, in case of equality,
$\mathcal{F}_i ={[n] \choose i}$ and $\mathcal{F}_{s+1-i}= \emptyset$ holds.
\end{lemma}

\subsection{The shifting operation and the lexicographic order}

The remainder of this section is devoted to a useful operation of families
 used in this article and a lemma related to lexicographic order.
Let us recall the definition of the $(i,j)$-shift $S_{i,j}$.
Given a family $\mathcal{F} \subseteq 2^{[n]}$,
for $1\le i<j\le n$, we define
\[  S_{i,j}(\mathcal{F}) = \{S_{i,j}(F) : F\in \mathcal{F}\}, \]
where
\[  S_{i,j}(F) = \begin{cases}
F':=(F\setminus \{j\}) \cup \{i\}, & \text{if $j\in F,i\notin F$
and $F'\notin \mathcal{F}$}; \\
F, & \text{otherwise}.
\end{cases} \]
This operation was introduced
 by Erd\H{o}s,
Ko and Rado \cite{EKR1961} and is now 
an important technique in extremal set theory;
see the comprehensive book \cite{FT2018}.
From the definition, we know that
$|S_{i,j}({F})| =|{F}|$
and $|S_{i,j}(\mathcal{F})| =|\mathcal{F}|$.
We say that $\mathcal{F}$ is {\it left-shifted} if
$S_{i,j}(\mathcal{F}) =\mathcal{F}$ for all
$1\le i<j \le n$.
The following fact is frequently used  in this paper.

\begin{fact} Let $\mathcal{F}$ be a left-shifted family
and $\{a_1,\ldots ,a_k\}$ and
$ \{b_1,\ldots ,b_k\}$
be two sets such that
$a_1< \cdots <a_k$ and $b_1< \cdots < b_k$.
If $a_i \le b_i$ for every $1\le i\le k$
and $\{b_1,\ldots ,b_k\}\in \mathcal{F}$,
then $\{a_1,\ldots ,a_k\}\in \mathcal{F}$.
\end{fact}

Let $\mathcal{A}$ and $\mathcal{B}$ be two  families
of  subsets of $[n]$.
We say that $\mathcal{A}$ and $\mathcal{B}$
 are {\it cross-intersecting} if
 $A\cap B\neq \emptyset$ for any $A\in \mathcal{A}$ and
 $B\in \mathcal{B}$.
Recently, the research on cross-intersecting families has attracted extensive attention; 
see, e.g., \cite{FW2023,SFQ2022}. 
It is well-known that applying the left-shifting operation $S_{i,j}$
on two cross-intersecting families $\mathcal{A}$ and $\mathcal{B}$,
the resulting two families are still cross-intersecting.

\begin{fact}[See \cite{EKR1961}]\label{factci}
Let $\mathcal{A}$ and $\mathcal{B}$ be  families
of  subsets of $[n]$. Let $\mathcal{A}$ be $t$-intersecting.
If $\mathcal{A}$ and $\mathcal{B}$ are cross-intersecting,
then $S_{i,j}(\mathcal{A})$ and $S_{i,j}(\mathcal{B})$ are also cross-intersecting.
Moreover, $S_{i,j}(\mathcal{A})$ is also $t$-intersecting.
\end{fact}

The following lemma follows from Fact \ref{factci}.

\begin{lemma}[See \cite{EKR1961}] \label{lem32}
Let $\mathcal{A} \subseteq {[n] \choose a}$
and $\mathcal{B} \subseteq {[n] \choose b}$ be
cross-intersecting families and let $\mathcal{A}$ be $t$-intersecting.
Then there exist left-shifted families $\mathcal{A}' \subseteq {[n] \choose a}$
and $\mathcal{B}' \subseteq {[n] \choose b}$ such that all of
 the following hold. 
 \begin{itemize}
 
 \item[{\em (i)}]
  $|\mathcal{A}|=|\mathcal{A}'|$ and $|\mathcal{B}|=|\mathcal{B}'|$;

\item[{\em (ii)}] 
 $\mathcal{A}'$ and $\mathcal{B}'$ are cross-intersecting;

\item[{\em (iii)}] 
 $\mathcal{A}'$ is $t$-intersecting. 
 \end{itemize}
\end{lemma}

Finally, let us define the lexicographic on the
$k$-element subsets of $[n]$.
We say that $F$ is smaller than $G$
in the lexicographic order, denoted by
$F\prec G$, if $\min (F\setminus G)< \min (G\setminus F)$ holds.
For example, $\{1,2,3\}\prec \{1,3,4\}$.
Let $k\in [0,n]$ and $m\in [0,{n \choose k}]$ be positive integers.
We  denote by
$\mathcal{L}(n,k,m)$ the family of the smallest $m$ sets
from ${[n] \choose k}$ in the lexicographic order.
Let $a,b,n$ be positive integers with $n>a+b$.
Hilton \cite{Hil1976} observed that
$\mathcal{A} \subseteq {[n] \choose a}$
and $\mathcal{B} \subseteq {[n] \choose b}$ are cross-intersecting
if and only if
$\mathcal{A} \cap \Delta_a (\mathcal{B}^c) = \emptyset$,
where $\mathcal{B}^c=\{[n]\setminus B : B\in \mathcal{B}\}$
denotes the family of complements of sets of $\mathcal{B}$.
This observation together with the
Kruskal--Katona theorem \cite{Kat1966,Kru1963} implies
the following lemma, which
 plays an important role in the treatment of cross-intersecting families;
see  \cite[p.266]{FK2017} for a detailed proof.

\begin{lemma}[See \cite{Hil1976,Kat1966,Kru1963}] \label{lem31}
Let $a,b,n$ be positive integers with $n>a+b$.
If $\mathcal{A} \subseteq {[n] \choose a}$
and $\mathcal{B} \subseteq {[n] \choose b}$ are  cross-intersecting,
then  $\mathcal{L}(n,a,|\mathcal{A}|)$
and $\mathcal{L}(n,b, |\mathcal{B}|)$ are cross-intersecting.
\end{lemma}

\section{A result for cross-intersecting families}

\label{sec3}

 The cross-intersecting property is a
 natural extension on the intersecting property.
In this section,
we shall  prove some important properties
of pairs of cross-intersecting families.

\begin{theorem}\label{cross-intersectingtheo}
Let $d\ge 3$ and $n\ge 2d+1$ be positive integers.
Let $\mathcal{A}\subseteq {[n]\choose d+1}$ and
$\mathcal{B}\subseteq {[n]\choose d}$ be  cross-intersecting families.
If  $|\mathcal{A}|\ge 2$ and $\mathcal{A}$ is $2$-intersecting, then
\begin{equation*} \label{eqq3.1}
|\mathcal{A}|+|\mathcal{B}|\le {n\choose d}-{n-d-1 \choose d}- {n-d-2 \choose d-1}+2.
\end{equation*}
For $n\ge  2d+2$, the above equality holds if and only if, under the isomorphism, $\mathcal{A}=\{[d+1], [d]\cup \{d+2\}\}$ and $\mathcal{B}=\{B\in {[n] \choose d}: B\cap [d] \neq \emptyset
~\text{or}~  \{d+1,d+2\} \subseteq B\}$;
 or two more possibilities when $d=3$, namely, 
$\mathcal{A}=\{A\in {[n] \choose 4}: [3]\subseteq A\}$ and $\mathcal{B}=\{B\in {[n] \choose 3}: B\cap [3]\neq \emptyset\}$; or 
$\mathcal{A}=\{A\in {[n] \choose 4}: [2]\subseteq A\}$ and 
$\mathcal{B}=\{B\in {[n] \choose 3}: B\cap [2] \neq \emptyset\}$.  
\end{theorem}


To prove Theorem \ref{cross-intersectingtheo},
we shall present a series of lemmas.
The following lemma
states that the condition $d\ge 3$ in Theorem \ref{cross-intersectingtheo}
is necessary since the result
is not true for the case $d=2$.

\begin{lemma}\label{cross-intersectinglemma2}
Let $n\ge 6$ be an integer and let $\mathcal{A}\subseteq {[n]\choose 3}$ and
$\mathcal{B}\subseteq {[n]\choose 2}$ be  cross-intersecting families.  If $|\mathcal{A}|\ge 2$ and $\mathcal{A}$ is $2$-intersecting, then
$$|\mathcal{A}|+|\mathcal{B}|\le {n\choose 2}-{n-3 \choose 2}+1.$$
Equality holds if and only if $\mathcal{A}=\{\{1,2\}\cup \{i\}: i\in[3,n]\}$ and $\mathcal{B}=\{B\in {[n] \choose 2}: 1\in B \:{\rm or} \: 2\in B\}$ under isomorphism.
\end{lemma}

\begin{proof} 
Since $\mathcal{A}$ is $2$-intersecting
and $n\ge 6$, by Theorem \ref{thm21},
we get $|\mathcal{A}|\le n-2$. 
By Lemma \ref{lem31},
we may assume that $\mathcal{A},\mathcal{B}$ are the collections of the smallest $|\mathcal{A}|,|\mathcal{B}|$ sets
in ${[n]\choose 3}, {[n]\choose 2}$ with respect to the lexicographical order, respectively. 
Then \[ \mathcal{A}=\Bigl\{ \{1,2,i\}: 3\le i\le  |\mathcal{A}| +2 \Bigr\}.\]

{\bf Case 1.} $|\mathcal{A}|=2$, that is, $\mathcal{A}=\{A_1, A_2\}$,
where $A_1=\{1,2,3\}$ and
$A_2=\{1,2,4\}$.
Denote by
\[ \mathcal{B}_1=
\left\{ B\in {[n] \choose 2}: B\cap A_1= \emptyset \right\},\]
and
 \[ \mathcal{B}_2=
 \left\{ B\in {[n] \choose 2}: B\cap A_1=\{3\},B\cap A_2= \emptyset  \right\}.\]
Since $\mathcal{A}, \mathcal{B}$ are cross-intersecting,
we have
$B\cap A_i\neq \emptyset$ for both $i=1, 2$ and every $B\in \mathcal{B}$.
Then the families $\mathcal{B}_1, \mathcal{B}_2$
and $\mathcal{B}$ are pairwise disjoint.
Therefore
\[  |\mathcal{B}| \le  {n \choose 2} - |\mathcal{B}_1|-|\mathcal{B}_2|
= {n\choose 2}-{n-3 \choose 2}- (n-4)
< {n\choose 2}-{n-3 \choose 2}-1. \]

{\bf Case 2.} $3\le |\mathcal{A}|\le n-2$.
Since $\mathcal{A}, \mathcal{B}$ are cross-intersecting, we get
$B\cap \{1,2\} \neq \emptyset$ for each $B\in \mathcal{B}$. Then
$ \mathcal{B}\subseteq \left\{B\in {[n]\choose 2}: B\cap [2]\neq \emptyset\right\}$.
Hence ,we have
\begin{align*}
|\mathcal{A}| + |\mathcal{B}|
\le n-2 + {n \choose 2} -{n-2 \choose 2}
= {n\choose 2}-{n-3 \choose 2}+1.
\end{align*}
Equality holds if and only if $\mathcal{A}=\{\{1,2\}\cup \{i\}: i\in[3,n]\}$ and $\mathcal{B}=\{B\in {n\choose 2}: 1\in B \:{\rm or}\: 2\in B\}$ under isomorphism.
Now we consider an arbitrary family $\mathcal{A}\subseteq {[n]\choose 3}$ with $|\mathcal{A}|=n-2$.
Since $\mathcal{A}$ is 2-intersecting, Theorem \ref{thm21} implies 
either $\mathcal{A}={T\choose 3}$ for some $T\in {[n]\choose 4}$ or
 $\mathcal{A}$ is isomorphic to $\{\{1,2\}\cup \{i\}: i\in[3,n]\}$.
 So assume that $\mathcal{A}={T\choose 3}$ for
 some $T\in {[n]\choose 4}$, which together with $|\mathcal{A}|=n-2$
 yields  $n=6$.
 Then $\mathcal{B}\subseteq {T\choose 2}$ since $\mathcal{A}$ and
$\mathcal{B}$ are cross-intersecting.
Then $|\mathcal{A}| + |\mathcal{B}|\le {4\choose 3}+{4\choose 2}=10<{6\choose 2}-{3 \choose 2}+1=13$.
So the extremal family is unique up to isomorphism. 
\end{proof}

\begin{lemma}\label{cross-intersectinglemma1}
Let $d\ge 1$ be an integer.
If $\mathcal{A}\subseteq {[2d+1]\choose d+1}$ and
$\mathcal{B}\subseteq {[2d+1]\choose d}$ are  cross-intersecting,
then
$$|\mathcal{A}|+|\mathcal{B}|\le {2d+1\choose d}.$$
\end{lemma}
\begin{proof}
For every $F\in {[2d+1]\choose d+1}$,
we denote  $F^c=[2d+1] \setminus F$.
 Since $\mathcal{A}$ and
$\mathcal{B}$ are cross-intersecting,
we have $F \notin \mathcal{A}$ or $F^c \notin \mathcal{B}$.
Thus,  at most half of the sets in ${[2d+1]\choose d+1}\cup {[2d+1]\choose d}$ belong to $\mathcal{A} \cup \mathcal{B}$.
Then $$|\mathcal{A}|+|\mathcal{B}|\le {1\over 2}\left({2d+1\choose d+1}+{2d+1\choose d}\right)={2d+1\choose d}.$$
This completes the proof. 
\end{proof}

The following result due to Frankl and Tokushige \cite{FT1992} is needed for our purpose. 
We remark that a generalization was recent proved by Frankl and Wang \cite{FW2024-EUJC}. 

\begin{lemma}[See \cite{FT1992}] \label{lem-FT-cross}
If $\mathcal{A}\subseteq {[n] \choose a}$ and 
$\mathcal{B}\subseteq {[n] \choose b}$ are non-empty 
cross intersecting families with $n\geqslant a+b$ and $a\leqslant b$,  
then  
\[ |\mathcal{A}| + |\mathcal{B}|\leqslant 
{n \choose b} - {n-a \choose b} +1.\] 
For $n>a+b$, the equality holds if and only if 
$\mathcal{A}= \{A\}$ 
and $\mathcal{B}=\{B\in {[n] \choose b}: B\cap A\neq \emptyset \}$ or for $(a,b)=(2,2)$, there is one more possible family 
$\mathcal{A}=\mathcal{B}=\{S\in {[n] \choose 2}: 1\in S\}$. 
\end{lemma}

Now we are ready to prove Theorem \ref{cross-intersectingtheo}.

\medskip

\noindent 
{\bf Proof of Theorem \ref{cross-intersectingtheo}}.
First let us apply induction on $n\ge 2d+1$ and $d \ge 3 $.
For the base case $n=2d+1$, the result holds by Lemma \ref{cross-intersectinglemma1}.
Assume that $n\ge 2d+2$ and the result holds for integers less than $n$.
By Lemma \ref{lem32}, we may assume that $\mathcal{A}$ and $\mathcal{B}$ are left-shifted.
Recall that
$\mathcal{A}(n):=\{A\setminus \{n\}: n\in
A\in \mathcal{A}\}$ and
$\mathcal{A}(\overline{n}) := \{A\in \mathcal{A}: n\notin A \}$.
For the family $\mathcal{B}$,
we define $\mathcal{B}(n)$ and $\mathcal{B}(\overline{n})$ similarly.

\medskip 
{\bf Claim 1.} $|\mathcal{A}(\overline{n})| \ge 2$. 

\begin{proof} 
If $\mathcal{A}(n)=\emptyset$, then $\mathcal{A}
=\mathcal{A}(\overline{n})$ and the claim holds
since $|\mathcal{A}| \ge 2$.
If $\mathcal{A}(n)\neq \emptyset$,
there exists $A\in \mathcal{A}$ with $n\in A$.
Since $n\ge 2d+2$ and $\mathcal{A}$ is left-shifted,
there exist two different elements
$u,v\in [n]\setminus A$ such that $(\mathcal{A}\setminus \{n\})\cup \{u\}, (\mathcal{A}\setminus \{n\})\cup \{v\}\in \mathcal{A}(\overline{n})$.
Thus the claim holds. 
\end{proof}

We can easily see that $\mathcal{A}(\overline{n})
$ and $\mathcal{B}(\overline{n}) $ are cross-intersecting, and  $\mathcal{A}(\overline{n})$
is $2$-intersecting.
Note that $\mathcal{A}(\overline{n})
\subseteq {[n-1] \choose d+1}$
and $\mathcal{B}(\overline{n}) \subseteq {[n-1 ] \choose d}$.
By applying the induction hypothesis, we have
\begin{equation} \label{eqqa}
|\mathcal{A}(\overline{n})|+|\mathcal{B}(\overline{n})|\le {n-1 \choose d}-{n-d-2 \choose d}- {n-d-3 \choose d-1}+2.
\end{equation}

In the sequel, we shall proceed in three cases: $\mathcal{A}(n)=\emptyset$, $|\mathcal{A}(n)|=1$ and
$|\mathcal{A}(n)|\ge 2$.

{\bf Case 1.} $\mathcal{A}(n)=\emptyset$.

Claim 1 gives $|\mathcal{A}(\overline{n})|\ge 2$.
Let $A_1, A_2 \in \mathcal{A}(\overline{n})$
and  $x\in [n-1]$ such that
$x\in A_1 \setminus A_2$.
 We denote 
\[ \mathcal{B}_1=
\left\{ B\in {[n-1] \choose d-1}: B\cap A_1= \emptyset \right\},\]
and
 \[ \mathcal{B}_2=
 \left\{ B\in {[n-1] \choose d-1}: B\cap A_1=\{x\},B\cap A_2= \emptyset  \right\}.\] 
 Clearly, we have $|\mathcal{B}_1|= {n-d-2 \choose d-1}$ and 
 $|\mathcal{B}_2| \ge {n-d-3 \choose d-2}$. 
 Note that $\mathcal{B}(n)$ and $\mathcal{A}(\overline{n})$ are cross-intersecting, and then $\mathcal{B}_1,\mathcal{B}_2$ and
$ \mathcal{B}(n)$ are pairwise disjoint. Therefore, we get
\begin{equation} \label{eqqb}
|\mathcal{B}(n)|\le
{n-1 \choose d-1} - |\mathcal{B}_1| - |\mathcal{B}_2| \le
 {n-1 \choose d-1}-{n-d-2 \choose d-1}-{n-d-3 \choose d-2}.
\end{equation}
Note that the equality in (\ref{eqqb}) holds if and only if
$\mathcal{A}(\overline{n})=\{A_1, A_2\}$ with
$|A_1\cap A_2|= d$ and $\mathcal{B}(n)
=\{B\in {[n-1] \choose d-1}: B\cap A_1 \neq \emptyset
~\text{and}~ B \cap A_2\neq \emptyset\}$.

Combining the inequalities (\ref{eqqa}) and (\ref{eqqb}), we get
 \begin{eqnarray} \label{eqeqc}
 |\mathcal{A}| + |\mathcal{B}|= |\mathcal{A}(\overline{n})|+|\mathcal{B}(\overline{n})|+|\mathcal{B}(n)|
 \le {n\choose d}-{n-d-1 \choose d}- {n-d-2 \choose d-1}+2.
\end{eqnarray}
The equality holds if and only if both equalities
 in (\ref{eqqa}) and (\ref{eqqb}) hold.
Hence, equality in (\ref{eqeqc}) holds  if and only if
$\mathcal{A}=\mathcal{A}(\overline{n}) =
\{A_1, A_2\}$ with $|A_1\cap A_2|= d$ and $\mathcal{B}=
\bigl\{ B\in {[n] \choose d}: B\cap A_1 \neq \emptyset
~\text{and}~ B \cap A_2\neq \emptyset \bigr\}$.

{\bf Case 2.} $|\mathcal{A}(n)|=1$, that is, there exists $A\in \mathcal{A}$ with $n\in A$.

For each $i=d+1,\ldots ,n$, we denote
$A_i=\{1,2,\ldots ,d,i\}$ and $\mathcal{A}'=\{A_{d+1}$, $A_{d+2}$, $\ldots $, $A_{n-1}\}$.
Since $\mathcal{A}$ is left-shifted and
$n\in A\in \mathcal{A}$,
we get $A_i\in \mathcal{A}$ for each $i$, that is,
$\mathcal{A}' \subseteq \mathcal{A}(\overline{n})$.
Note that $\mathcal{A}$ and
$ \mathcal{B}$ are cross-intersecting,
 which implies that
 $\mathcal{A}'$ and $\mathcal{B}(n)$ are cross-intersecting.
 Since $n\ge 2d+2$, we have $|\mathcal{A}'|\ge d+1$.
Then $B\cap [d]\neq \emptyset $ for every $B\in \mathcal{B}(n)$.
So
\begin{equation} \label{eqqc}
|\mathcal{B}(n)|\le {n-1\choose d-1}-{n-d-1 \choose d-1}.
\end{equation}
The equality holds if and only if $\mathcal{B}(n)
=\bigl\{ B\in {[n-1] \choose d-1} :
B\cap [d] \neq \emptyset \bigr\}$ and  then
$\mathcal{A}'=\mathcal{A}(\overline{n})$.

Firstly, we assume that $\mathcal{A}'=\mathcal{A}(\overline{n})$.
Since $|\mathcal{A}(n)|=1$ and $\mathcal{A}$ is left-shifted, we have
$\mathcal{A}=\{[d]\cup \{i\}:i \in [d+1,n]\}$.
Then $\mathcal{B}\subseteq \{B\in {[n]\choose d}: B\cap [d]\neq \emptyset\}$.
Hence
 \begin{eqnarray*}
 |\mathcal{A}| + |\mathcal{B}|&\le & (n-d)+ {n\choose d}-{n-d\choose d}\\
 &=& {n\choose d}-{n-d-1 \choose d}- {n-d-2 \choose d-1}+2+\left(n-d-2-{n-d-2\choose d-2}\right) \\
 &\le& {n\choose d}-{n-d-1 \choose d}- {n-d-2 \choose d-1}+2.
\end{eqnarray*}
The equality holds if and only if $n-d-2
= {n-d-2\choose d-2}$, which leads to
$d=3$,
and  $\mathcal{A}=\bigl\{ \{1,2,3\}\cup \{i\} :
i \in [4,n]\bigr\}$ and
$\mathcal{B}= \{ B\in {[n] \choose 3}: B\cap [3] \neq \emptyset\}$.

Now, we assume that $\mathcal{A}'\neq \mathcal{A}(\overline{n})$.
Then the inequality in (\ref{eqqc}) holds strictly, that is, 
\begin{equation} \label{neqq}
 |\mathcal{B}(n)|< {n-1\choose d-1}-{n-d-1 \choose d-1}.
\end{equation}
So combining (\ref{eqqa}) and (\ref{neqq}), and $|\mathcal{A}(n)|=1$, we have
 \begin{eqnarray*}
 |\mathcal{A}| + |\mathcal{B}|&=& |\mathcal{A}(\overline{n})|+|\mathcal{B}(\overline{n})|+|\mathcal{A}(n)|+|\mathcal{B}(n)| \\
 &<&
  {n-1 \choose d} \!-\! {n\!-\!d\!-\!2 \choose d} \!-\! {n\!-\!d\!-\!3 \choose d-1}+2+1+{n-1\choose d-1}-{n \!-\!d\!-\!1 \choose d-1}  \\
 &=& {n\choose d}-{n-d-1 \choose d}- {n-d-2 \choose d-1}+2+\left(1-{n-d-3\choose d-3}\right) \\
 &\le& {n\choose d}-{n-d-1 \choose d}- {n-d-2 \choose d-1}+2.
\end{eqnarray*}

{\bf Case 3.} $|\mathcal{A}(n)|\ge 2$. 
In this case, we shall prove the following two claims. 

{\bf Claim 2.} $\mathcal{A}(n)$ is $2$-intersecting.
Suppose on the contrary  that there exist $A_1, A_2\in \mathcal{A}(n)$ such that $|A_1\cap A_2|\le 1$.
Note that $\mathcal{A}$ is $2$-intersecting,
we then have $|A_1\cap A_2|=1$.
Thus $|A_1\cup A_2|=|A_1|+|A_2|-|A_1\cap A_2|= 2d-1$.
Since $n\ge 2d+2$, there exist two distinct elements
$u,v\in [n]\setminus (A_1\cup A_2)$.
Since $A_1,A_2\in \mathcal{A}(n)$, we have 
$A_1\cup \{n\}$ and $A_2\cup \{n\}$ are contained
in $\mathcal{A}$.
Then both $A_1\cup \{u\}$ and $A_2\cup \{v\}$ belong to $\mathcal{A}$ as $\mathcal{A}$ is left-shifted.
We can see that  $|(A_1\cup \{u\}) \cap  (A_2\cup \{v\})|=1$, contradicting the fact that $\mathcal{A}$ is $2$-intersecting.


{\bf Claim 3.} $\mathcal{A}(n)$ and $\mathcal{B}(n)$ are cross-intersecting.
For the sake of a contradiction, 
suppose that there exist $A \in \mathcal{A}(n)$ and $B \in \mathcal{B}(n)$ such that $A\cap B=\emptyset$.
Since  $|A\cup B|=|A|+|B|=2d-1$ and $n\ge 2d+2$,
there exist two different elements $u,v\in [n]\setminus (A\cup B)$.
Note that $A\cup \{n\}\in \mathcal{A}$, $B\cup \{n\}\in \mathcal{B}$.
Then $A\cup \{u\}\in \mathcal{A}$ and $B\cup \{v\}\in \mathcal{B}$ since $\mathcal{A}$ and $\mathcal{B}$ are shifted.
We can observe that
 $(A\cup \{u\}) \cap (B\cup \{v\})=\emptyset$, contradicting that $\mathcal{A}$ and $\mathcal{B}$ are cross-intersecting.

Let us apply the induction on $d$. For the base case $d=3$,
we have
$\mathcal{A}(n) \subseteq {[n-1] \choose 3}$ and
$\mathcal{B}(n) \subseteq {[n-1] \choose 2}$.
 If  $\mathcal{A}(n)$ is not isomorphic to
 $\bigl\{ \{1,2\}\cup \{i\}: i\in[3,n-1] \bigr\}$,
then the inequality in Lemma \ref{cross-intersectinglemma2} is strict, and we get
\begin{equation} \label{eqqd}
|\mathcal{A}(n)|+|\mathcal{B}(n)| <  {n-1 \choose 2}-{n-4\choose 2}+1.
\end{equation}
Adding (\ref{eqqa}) and (\ref{eqqd}), we have
\begin{eqnarray*}
 |\mathcal{A}| + |\mathcal{B}|&=& |\mathcal{A}(\overline{n})|+|\mathcal{B}(\overline{n})|+|\mathcal{A}(n)|+|\mathcal{B}(n)| \\
 &<& \left({n-1 \choose 3}-{n-5 \choose 3}- {n-6 \choose 2}+2 \right) +\left({n-1 \choose 2}-{n-4\choose 2}+1 \right)  \\
 &=& {n\choose 3}-{n-4 \choose 3}- {n-5 \choose 2}+2.
\end{eqnarray*} 
Next, we assume under isomorphism
 that $\mathcal{A}(n)=\bigl\{ \{1,2\}\cup \{i\}: i\in[3,n-1] \bigr\}$.
Thus, we have $\{1,2,n-1,n\}\in \cal A$. 
Since $\cal A$ is left-shifted, it follows that 
$\mathcal{A}$ contains every set $\{1,2,i,j\}$ with $3\le i<j\le n$. 
Consequently, we get $\mathcal{A}= \{A\in {[n] \choose 4}: [2]\subseteq A\}$. Otherwise, if $\{1,3,4,5\}\in \cal A$, 
then $|\{1,3,4,5\}\cap \{1,2,6,7\}|=1$, 
which is a contradiction since $\cal A$ is 2-intersecting.
Note that each set  $B\in \mathcal{B}$ satisfies
 $B\cap [2] \neq \emptyset$. 
Hence 
\[ |\mathcal{A}| + |\mathcal{B}|\le  \binom{n-2}{2}+ 
\left(\binom{n}{3} - \binom{n-2}{3}\right)= {n\choose 3}-{n-4 \choose 3}- {n-5 \choose 2}+2.\]
Moreover, the above equality holds if and only if 
$\mathcal{A}=\{A\in {[n] \choose 4}: [2]\subseteq A\}$ and 
$\mathcal{B}=\{B\in {[n] \choose 3}: B\cap [2] \neq \emptyset\}$ under the isomorphism.

So assume that $d\ge 4$ and the result holds for less than $d$.
Note that $\mathcal{A}(n) \subseteq {[n-1] \choose d}$ and
$\mathcal{B}(n) \subseteq {[n-1] \choose d-1}$.
Then applying the induction hypothesis, we have
\begin{eqnarray*}
|\mathcal{A}(n)|+|\mathcal{B}(n)|&\le& {n-1 \choose d-1}-{(n-1)-(d-1)-1 \choose d-1}- {(n-1)-(d-1)-2 \choose d-2}+2 \\
&=& {n-1 \choose d-1}-{n-d-1 \choose d-1}- {n-d-2 \choose d-2}+2,
\end{eqnarray*}
which together with the inequality (\ref{eqqa}) yields
\begin{small}
 \begin{eqnarray*}
 |\mathcal{A}| + |\mathcal{B}|&=& |\mathcal{A}(\overline{n})|+|\mathcal{B}(\overline{n})|+|\mathcal{A}(n)|+|\mathcal{B}(n)| \\
 &\le& {n\choose d}-{n-d-1 \choose d}- {n-d-2 \choose d-1}+2+\left(2-{n-d-2 \choose d-2}-{n-d-3 \choose d-3}\right)  \\
 &<& {n\choose d}-{n-d-1 \choose d}- {n-d-2 \choose d-1}+2,
\end{eqnarray*}
\end{small}
where $2-{n-d-2 \choose d-2}-{n-d-3 \choose d-3}<0$ since $n\ge 2d+2$ and $d\ge 3$.

{
In the above discussion, we have determined the extremal families 
$\mathcal{A}$ and $\mathcal{B}$ that attain the required upper bound  under the left-shifting assumption by Lemma \ref{lem32}. 
Next, we are going to characterize the extremal families in general case. Let $\mathcal{A}_0$ and $\mathcal{B}_0$ be the initial families before applying the left-shifting operations. 
In other words, $\mathcal{A}, \mathcal{B}$ are obtained from $\mathcal{A}_0, \mathcal{B}_0$ by applying a series of left-shifting operations. In what follows, 
we show that $\mathcal{A}_0$ is isomorphic to $\mathcal{A}$, and $\mathcal{B}_0$ is isomorphic to $\mathcal{B}$ as well. 

 Suppose on the contrary that there exist two families $\mathcal{A}_1$ and $\mathcal{B}_1$ such that $\mathcal{A}_1$ is not isomorphic to $\mathcal{A}$, and $\mathcal{B}_1$ is not isomorphic to $\mathcal{B}$ with $S_{i,j}(\mathcal {A}_1)=\cal A$ and $S_{i,j}(\mathcal {B}_1)=\cal B$ for some $1\le i<j\le n$.
We proceed the argument in the following three cases.

{\bf Case 1}. $\mathcal{A}=\{[d+1], [d]\cup \{d+2\}\}$ and $\mathcal{B}=\{B\in {[n] \choose d}: B\cap [d] \neq \emptyset
~\text{or}~  \{d+1,d+2\} \subseteq B\}$.
We denote $\mathcal {A}_1=\{A_1,A_2\}$. Since $\mathcal{A}_1$ is not isomorphic to $\mathcal{A}$, we have $|A_1\cap A_2|\le d-1$. 
As $\mathcal{A}_1$ and $\mathcal{B}_1$ are cross-intersecting, 
we get $|\mathcal{B}_1|<|\mathcal{B}|$, contradicting to $|\mathcal{B}_1|=|\mathcal{B}|$.

{\bf Case 2}. $\mathcal{A}=\{A\in {[n] \choose 4}: [3]\subseteq A\}$ and $\mathcal{B}=\{B\in {[n] \choose 3}: B\cap [3]\neq \emptyset\}$.
Recall that $S_{i,j}(\mathcal{A}_1)= \mathcal{B}$ and $S_{i,j}(\mathcal{B}_1)= \mathcal{B}$. 
Then it yields that $j\in [4,n]$ and $i\in [3]$. Without loss of generally, we may assume that $j=4$ and $i=3$.
We denote $\mathcal{A}_1=\{[4]\} \cup \mathcal{A}_{11}\cup \mathcal{A}_{12}$ and $\mathcal{B}_1=\{B\in {[n] \choose 3}: B\cap [2]\neq \emptyset, \:{\rm or} \:[3,4]\subset A\} \cup \mathcal{B}_{11}\cup \mathcal{B}_{12}$, where $\mathcal{A}_{11}=\{A\in \mathcal{A}_{1}:3\in A, 4\notin A\}$, $\mathcal{A}_{12}=\{A\in \mathcal{A}_{1}:3\notin A, 4 \in A\}$,
$\mathcal{B}_{11}=\{A\in \mathcal{B}_{1}\cap \binom{[3,n]}{3}:3\in A, 4\notin A\}$ and $\mathcal{B}_{12}=\{B\in \mathcal{B}_{1}\cap \binom{[4,n]}{3}: 4 \in B\}$. 
Note that 
$\mathcal{A}_{11}\neq \emptyset$ and $\mathcal{A}_{12}\neq \emptyset$ since $\mathcal{A}_1$ is not isomorphic to $\mathcal{A}$.
Similarly, we have $\mathcal{B}_{11}\neq \emptyset$ and $\mathcal{B}_{12}\neq \emptyset$.
We know that $\mathcal{A}_{11}$ and $\mathcal{B}_{12}$ are cross-intersecting, $\mathcal{A}_{12}$ and $\mathcal{B}_{11}$ are cross-intersecting. 
We denote $\mathcal{A}_{11}'=\{A\setminus [3]:A\in \mathcal{A}_{11}\}$ and $\mathcal{B}_{12}'=\{B\setminus \{4\}:B\in \mathcal{B}_{12}\}$.
It is clear that $\mathcal{A}_{11}'\subseteq \binom{[5,n]}{1}$ and $\mathcal{B}_{12}'\subseteq \binom{[5,n]}{2}$ are non-empty cross-intersecting with $|\mathcal{A}_{11}'|=|\mathcal{A}_{11}|$ and $|\mathcal{B}_{12}'|=|\mathcal{B}_{12}|$. 
Hence, Lemma \ref{lem-FT-cross} implies  
\[|\mathcal{A}_{11}|+|\mathcal{B}_{12}|=|\mathcal{A}_{11}'|+|\mathcal{B}_{12}'|\le \binom{n-4}{2}-\binom{n-5}{2}+1.\]
Similarly, we have 
\[|\mathcal{A}_{12}|+|\mathcal{B}_{11}| \le \binom{n-4}{2}-\binom{n-5}{2}+1.\]
Consequently, we obtain 
 \begin{eqnarray*}
 |\mathcal{A}_1|+|\mathcal{B}_1|&=&1+\binom{n-1}{2}+\binom{n-2}{2}+(|\mathcal{A}_{11}|+|\mathcal{B}_{12}|)+(|\mathcal{A}_{12}|+|\mathcal{B}_{11}|) \\
&\le & 1+\binom{n-1}{2}+\binom{n-2}{2}+2\left(\binom{n-4}{2}-\binom{n-5}{2}+1\right) \\
&< & (n-3)+\binom{n-1}{2}+\binom{n-2}{2}+\binom{n-3}{2} \\
&=& |\mathcal{A}|+|\mathcal{B}|,
\end{eqnarray*}
which leads to a contradiction.

{\bf Case 3}. $\mathcal{A}=\{A\in {[n] \choose 4}: [2]\subseteq A\}$ and 
$\mathcal{B}=\{B\in {[n] \choose 3}: B\cap [2] \neq \emptyset\}$.

This case is similar to Case 2, so we omit the details.
This completes the proof.
}
\qed

\section{Proof of Theorem \ref{thmmain}}

\label{sec4}

Let $4\le s \le n-2$ be integers,
$\mathcal{F}\subseteq 2^{[n]}$ be $s$-union, and $\mathcal{F}\nsubseteq \mathcal{K}(n,s)$
and $\mathcal{F}\nsubseteq \mathcal{H}(n,s)$.
Note that for each  $F\in \mathcal{F}$ and $E\subseteq F$,
the family
$\mathcal{F'}:=\mathcal{F} \cup \{E\}$  also has the $s$-union property.
Moreover, it is easy to see that $\mathcal{F}'$ is still not contained in the Katona family $\mathcal{K}(n,s)$
and the Frankl family $\mathcal{H}(n,s)$.
Thus, we may assume that $\mathcal{F}$ is hereditary (also known as a down-closed family, or complex), that is, $E\subseteq F\in \mathcal{F}$
implies $E\in \mathcal{F}$.
We next present our proof in two cases.

{\bf Case I.} Assume that $s=2d$.
Since $\mathcal{F}$ is $s$-union, we have $\mathcal{F}_i=\emptyset$ for every $i\ge 2d+1$.
Moreover, we can see that
$\mathcal{F}_{d+1}$ is 2-intersecting. Otherwise there exist two sets $A, B\in \mathcal{F}_{d+1}$ such that $|A\cap B|\le 1$,
then by the inclusion-exclusion principle, $|A\cup B|=|A|+|B|-|A\cap B|\ge 2d+1>s$, which contradicts the fact that $\mathcal{F}$ is $s$-union.

{\bf Subcase 1.1.} $d=2$.
In this case, we further assume that $\mathcal{F} \nsubseteq
\mathcal{H}^*(n,4)$.
If $\mathcal{F}$ has a $4$-element set $\{a,b,c,d\}$, then
$2^{\{a,b,c,d\}} \subseteq \mathcal{F}$ since $\mathcal{F}$ 
is hereditary. Since $s=4$ and $\mathcal{F}$ is $4$-union, we get 
$\mathcal{F} = 2^{\{a,b,c,d\}}$, and then 
$|\mathcal{F}| =2^4 < \sum_{i=0}^2 {n\choose i} - 
{n-3 \choose 2} - {n-4 \choose 2} +2$ since $n\ge s+2=6$, as required.
 Next, we  assume that $\mathcal{F}_4$ is empty.
 Since $\mathcal{F}$ is neither isomorphic to a subfamily of
 $ \mathcal{K}(n,4)$ nor of 
 $\mathcal{H}(n,4)$, we have
 $|\mathcal{F}_3| \ge 2$.
 If $\mathcal{F}_3$ has exactly two $3$-element sets,
 say $D_1,D_2$,  then 
 $|D_1\cap D_2|=2$. Since 
 $\mathcal{F}_3 $ and $\mathcal{F}_2$ are cross-intersecting, 
 we get $|\mathcal{F}_2| \le  {n \choose 2} - {n-4 \choose 2} - 2{n-4 \choose 1} =
 {n \choose 2} - {n -3 \choose 2}
 - {n-4 \choose 1}$. The Katona inequality (\ref{eqkat})  states that
 $ |\mathcal{F}_0| \le 1$ and $|\mathcal{F}_1| + |\mathcal{F}_4| \le n$.
 Thus, we get
 \[ |\mathcal{F}| \le
 \sum_{0\le i \le 2} {n \choose i} - {n -3 \choose 2}
 - {n-4 \choose 1} +2 .\]
 Moreover, the equality holds if and only if
 $\mathcal{F}=\mathcal{W}(n,4)$.
If $|\mathcal{F}_3|= 3$,
then there are two possibilities, namely,
$\mathcal{F}_3=\{\{a,b,c\},\{a,b,d\},\{a,b,e\}\}$
or $\mathcal{F}_3=\{\{a,b,c\}, \{a,b,d\}$, $\{b,c,d\}\}$
for some $\{a,b,c,d,e\} \subseteq [n]$.
For the first possibility,
we have $\mathcal{F}_2 \subseteq \{H \in {[n] \choose 2}:
H\cap \{a,b\} \neq \emptyset\}$, which implies that
$\mathcal{F}$ is isomorphic to a subfamily of
$\mathcal{H}^*(n,4)$, a contradiction;
for the second possibility,
we can verify that $\mathcal{F}_2
\subseteq \{H \in {[n] \choose 2} : b\in H\} \cup 
\bigl\{\{a,c\},\{a,d\},\{c,d\}\bigr\}$,
which leads to $|\mathcal{F}_2| \le n+2$.
Hence, we have $|\mathcal{F}| \le 1+n + (n-1) + 3$,
which is smaller than the required bound.
If $|\mathcal{F}_3| \ge 4$, then $\mathcal{F}_3
\subseteq \{\{a,b,x\} : x \in [n]\setminus \{a,b\}\}$,
which yields  $\mathcal{F}_2 \subseteq \{H \in {[n] \choose 2}:
H\cap \{a,b\} \neq \emptyset\}$. Hence
$\mathcal{F}$ is isomorphic to a subfamily of
$\mathcal{H}^*(n,4)$, a contradiction.

{\bf Subcase 1.2.} $d\ge 3$.
Since $\mathcal{F}$ is not a subfamily of $\mathcal{H}(n,2d)$ or $\mathcal{K}(n,2d)$, 
and $\mathcal{F}$ is hereditary, we get
$|\mathcal{F}_{d+1}| \ge 2$.
Setting $\mathcal{A}=\mathcal{F}_{d+1}$
and $\mathcal{B}=\mathcal{F}_d$, by Theorem \ref{cross-intersectingtheo},  we  obtain
\begin{equation} \label{eqeq5}
 |\mathcal{F}_d|+|\mathcal{F}_{d+1}|
\le  {n\choose d}-{n-d-1 \choose d}- {n-d-2 \choose d-1}+2.
\end{equation}
By Katona's inequality (\ref{eqkat}), we get that for each $i=0,1,\ldots ,d-1$,
\begin{equation}  \label{eqeq6}
 |\mathcal{F}_i|+|\mathcal{F}_{2d+1-i}| \le
{n \choose i}.
\end{equation}
Then adding (\ref{eqeq5}) and (\ref{eqeq6}), we have
 \begin{eqnarray*}
 |\mathcal{F}|&=&
 \sum_{0\le i \le d-1} (|\mathcal{F}_i|+|\mathcal{F}_{2d+1-i}|)  +
(|\mathcal{F}_d|+|\mathcal{F}_{d+1}|) \\
&\le& \sum_{0\le i \le d}{n\choose i}-{n-d-1 \choose d}- {n-d-2 \choose d-1}+2.
\end{eqnarray*}
For $d\ge 3$ and $n> 2d+1$,
whenever $|\mathcal{F}|$ attains the upper bound, 
 equalities have to hold in 
 (\ref{eqeq5}) and (\ref{eqeq6}) as well. Then 
$\mathcal{F}_i= {[n] \choose i}$ for every $i \le d-1$,
and $\mathcal{F}_{d+1}=\{D_1,D_2\}$ for fixed
$D_1,D_2 \in {[n] \choose d+1}$ with
$|D_1\cap D_2| =d$ and $\mathcal{F}_{d}=\{B\in {[n] \choose d}: B\cap D_1 \neq \emptyset
~\text{and}~ B \cap D_2\neq \emptyset\}$. 
{In addition,
for the case $d=3$, we know from Theorem \ref{cross-intersectingtheo} that 
there are two more possible families  
attaining the equality in (\ref{eqeq5}), namely, 
$ \mathcal{F}_3=\{F\in {[n] \choose 3} :
F\cap [3] \neq \emptyset \}$
and $\mathcal{F}_4=\bigl\{ F\in {[n] \choose 4}:[3] \subseteq F \bigr\}$; 
or 
$ \mathcal{F}_3=\{F\in {[n] \choose 3} :
F\cap [2] \neq \emptyset \}$
and $\mathcal{F}_4=\bigl\{ F\in {[n] \choose 4}: [2] \subseteq F\bigr\}$. 
Thus, $\mathcal{F}$ is isomorphic to $\mathcal{W}(n,6)$ or
$\mathcal{W}^*(n,6)$ or $W^{**}(n,6)$. 
}

{\bf Case II.} Assume that $s=2d+1$.
Since $\mathcal{F}$ is $s$-union, we know that $\mathcal{F}_{d+1}$ is intersecting.

{\bf Subcase 2.1.} $|\mathcal{F}_{d+2}|=0$.
By the hereditary property of $\mathcal{F}$, we get
 $\mathcal{F}_{i}=\emptyset$ for every $i\ge d+2$.
Since $\mathcal{F}\nsubseteq \mathcal{K}(n,2d+1)$ and $\mathcal{F}\nsubseteq \mathcal{H}(n,2d+1)$,
we have $\mathcal{F}_{d+1}\nsubseteq \mathcal{EKR}(n,d+1)$ and $\mathcal{F}_{d+1}\nsubseteq \mathcal{HM}(n,d+1)$.
For the case $d=2,s=5$, the condition
$\mathcal{F} \nsubseteq \mathcal{T}(n,5)$
implies that $\mathcal{F}_3 \nsubseteq
\mathcal{G}_2(n,3)$.
By Theorem \ref{thmhk}, we have
\begin{equation}  \label{eq5}
 |\mathcal{F}_{d+1}|\le {n-1 \choose d}-{n-d-2 \choose d}- {n-d-3 \choose d-1}+2.
 \end{equation}
Moreover, by Katona's inequality (\ref{eqkat}), then for each $i\in \{1,2,\ldots ,d\}$, we have
\begin{equation} \label{eq6}
|\mathcal{F}_i| + |\mathcal{F}_{2d+2-i}| \le {n \choose i}.
\end{equation}
Therefore, we obtain
 \begin{eqnarray*}\label{case1}
 |\mathcal{F}|&=&\sum_{0\le i \le d} (|\mathcal{F}_i|+|\mathcal{F}_{2d+2-i}|)+ |\mathcal{F}_{d+1}| \\
 &\le& \sum_{0\le i \le d}{n\choose i}+{n-1 \choose d}-{n-d-2 \choose d}- {n-d-3 \choose d-1}+2,
\end{eqnarray*}
In case of equality, the inequality in (\ref{eq5}) and (\ref{eq6}) must be an equality.
By Theorem \ref{thmhk}, when  $s\ge 9$, that is, $d\ge 4$, we get that
$\mathcal{F}_{d+1}$ is isomorphic to
$\mathcal{J}_2(n,d+1)$. Hence $\mathcal{F}$
is isomorphic to
\[ \mathcal{W}(n,2d+1)=\{F \subseteq [n]: |F|\le d\}
\cup \mathcal{J}_2(n,d+1).  \]
In addition, when $s=7$, that is, $d= 3$, then
$\mathcal{F}_{4}$ is isomorphic to
$\mathcal{J}_2(n,4)$, or $\mathcal{G}_2(n,4)$ or $\mathcal{G}_3(n,4)$.
Thus, $\mathcal{F}$
is isomorphic to
$ {[n] \choose \le 3} \cup \mathcal{J}_2(n,4)$, or
${[n] \choose \le 3} \cup \mathcal{G}_2(n,4)$
or ${[n] \choose \le 3} \cup \mathcal{G}_3(n,4)$;
when  $s=5$, that is, $d=2$, we see that 
$\mathcal{F}_{3}$ is isomorphic to
$\mathcal{J}_2(n,3)$, and then $\mathcal{F}$  is isomorphic to
${[n] \choose \le 2}\cup \mathcal{J}_2(n,3)$.

{\bf Subcase 2.2.}
$|\mathcal{F}_{d+2}|=1$.
Denote $\mathcal{F}_{d+2}=\{A\}$.
Note that ${A \choose d+1} \subseteq \mathcal{F}_{d+1}$ and then
$\cap_{F\in \mathcal{F}_{d+1}}F=\emptyset$.
By the Hilton--Milner theorem, we get
\[ |\mathcal{F}_{d+1}|\le {n-1 \choose d}-{n-d-2 \choose d}+1. \]
Since $\mathcal{F}$ is $(2d+1)$-union,
every set of $\mathcal{F}_d$ intersects  $A$,
and so $|\mathcal{F}_d|\le {n \choose d}-{n-d-2 \choose d}$.
Hence
\[ |\mathcal{F}_d|+|\mathcal{F}_{d+2}|\le {n \choose d}-{n-d-2 \choose d}+1.\]
Therefore, we obtain
 \begin{eqnarray*}
 |\mathcal{F}|&=&\sum_{0\le i \le d} (|\mathcal{F}_i|+|\mathcal{F}_{2d+2-i}|)+ |\mathcal{F}_{d+1}|\\
 &\le& \sum_{0\le i \le d-1}{n\choose i}+\left({n \choose d}-{n-d-2 \choose d}+1\right)+{n-1 \choose d}-{n-d-2 \choose d}+1\\
 &<& \sum_{0\le i \le d}{n\choose i}+{n-1 \choose d}-{n-d-2 \choose d}- {n-d-3 \choose d-1}+2.
\end{eqnarray*}

{\bf Subcase 2.3.}
$|\mathcal{F}_{d+2}|\ge 2$.
Since $\mathcal{F}$ is $s$-union, $\mathcal{F}_{d+1}$ is intersecting.
Furthermore, we claim that $\mathcal{F}_{d+1} \nsubseteq \mathcal{EKR}(n,d+1)$ and $\mathcal{F}_{d+1}\nsubseteq \mathcal{HM}(n,d+1)$.
Let $A, B\in \mathcal{F}_{d+2}$.
The down-closed property implies that both ${A \choose d+1}$ and $ {B \choose d+1}
$ are contained in $ \mathcal{F}_{d+1}$.
Firstly, we can see that
$\bigcap_{F\in \mathcal{F}_{d+1}} F= \emptyset$
since ${A \choose d+1}\subseteq \mathcal{F}_{d+1}$. Thus,
we get $\mathcal{F}_{d+1} \nsubseteq \mathcal{EKR}(n,d+1)$. Secondly,
we claim that $\mathcal{F}_{d+1}\nsubseteq \mathcal{HM}(n,d+1)$.
Assume on the contrary  that
$\mathcal{F}_{d+1}\subseteq \mathcal{HM}(n,d+1)$, where 
\[ \mathcal{HM}(n,d+1):=\left\{F \in {[n] \choose d+1}:
1\in F, F\cap  [2,d+2]\neq \emptyset \right\} \cup \bigl\{ [2,d+2] \bigr\}. \]
If $1\notin A$ or $1\notin B$, then
we can find at least two distinct $(d+1)$-element sets in
$ \mathcal{F}_{d+1}$ such that they  do not contain element $1$, contradicting  $\mathcal{F}_{d+1} \subseteq \mathcal{HM}(n,d+1)$.
If $1\in A\cap B$, then  both $A\setminus \{1\}$ and $B\setminus \{1\}$
are distinct $(d+1)$-element sets in $\mathcal{F}_{d+1}$
not containing element $1$,  a contradiction.
Thus $\mathcal{F}_{d+1}\nsubseteq \mathcal{HM}(n,d+1)$.
To apply Theorem \ref{thmhk} for the family $\mathcal{F}_{d+1}$ 
in the case $d=2$ and $s=5$,
we need to distinguish the cases 
whether $\mathcal{F}_3$ is a subfamily of $\mathcal{G}_2(n,3)$.

Firstly, we assume that $d\ge 3$, or $d=2$ and 
$\mathcal{F}_{3} \nsubseteq \mathcal{G}_2(n,3)$.
Then as in Subcase 2.1,
applying Theorem \ref{thmhk} and Katona's inequality (\ref{eqkat}) 
yields  the required inequality.
Furthermore, since $\mathcal{F}_{d+2}\neq \emptyset$, the equality
in (\ref{eq6}) can not  happen.
Thus, the required inequality holds strictly.

Now assume that $d=2$ and $\mathcal{F}_{3} \subseteq \mathcal{G}_2(n,3)$.
Note that $\mathcal{F}$ is $5$-union and
so $\mathcal{F}_4$ is $3$-intersecting.
Without loss of generality we may
assume that $\{1,2,3,4\}$, 
$\{1,2,3,5\} \in \mathcal{F}_4$.
Thus $\{1,2,3\}$, $\{1,2,4\}$, $\{1,3,4\}$, 
$\{2,3,4\}$, $\{1,2,5\}$,
$\{1,3,5\}$, $\{2,3,5\} \in \mathcal{F}_3$ since $\mathcal{F}$ is hereditary.
Then $\mathcal{F}_3\subseteq \{F\in {[n] \choose 3} : |F\cap \{1,2,3\}| \ge 2\}$ and
$\mathcal{F}_4 \subseteq\{F\in {[n] \choose 4} : \{1,2,3\} \subseteq F\}$.
If $|\mathcal{F}_{4}|\ge 3$, then $\mathcal{F}_2\subseteq \{F\in {[n] \choose 2} : |F \cap \{1,2,3\}| \ge 1\}$.
If $|\mathcal{F}_{4}|=2$, then $\mathcal{F}_2\subseteq \{F\in {[n] \choose 2} : |F \cap \{1,2,3\}| \ge 1\}\cup \{4,5\}$.
This implies that
\[ |\mathcal{F}_2| +|\mathcal{F}_4| \le
\max\{(3n-5)+2,(3n-6) + (n-3)\}=4n -9. \quad
  \]
The Katona inequality (\ref{eqkat}) implies that
$ |\mathcal{F}_0| \le 1$ and
$ |\mathcal{F}_1| +|\mathcal{F}_5| \le { n }$.
And   $|\mathcal{F}_3| \le  3n-8$ since $\mathcal{F}_3\subseteq \{F\in {[n] \choose 3} : |F\cap \{1,2,3\}| \ge 2\}$.
Therefore, we obtain $|\mathcal{F}| \le 8n -16$,
which is less than the desired bound
$1+ {n \choose 1} +{n \choose 2} +{n-1 \choose 2} - {n-4 \choose 2} - {n-5 \choose 1} +2=\frac{1}{2}(n^2+5n-2)$
since $n\ge s+2=7$.
This completes the proof.
\qed

\section{Concluding remarks}

For two subsets $A$ and $B$,
the symmetric difference of $A,B$ is defined as
\[  A\triangle B
=(A\setminus B)\cup (B\setminus A)=
(A\cup B) \setminus (A\cap B). \]
The distance of $A$ and $B$ is defined as
$d(A,B)=|A\triangle B|$.
The family of all subsets
with the binary function $d$ forms a metric space.
The diameter of a family $\mathcal{F}$ is defined as
the maximum distance of pairs of sets in $\mathcal{F}$.
Observe that $|F\cup F'| \le s$
implies $d(F,F')\le s$. However, the reverse  is
not true. 
In 1966, Kleitman \cite{Kle1966} extended
 Katona's Theorem \ref{thm1.1} to the families with given diameter.

\begin{theorem}[Kleitman \cite{Kle1966}] \label{thm-Kleit}
Suppose that $n>s \ge 2$ and
$\mathcal{F}$ is a family of subsets of $[n]$ with diameter at most $s$,
that is,  $d(F,F')\le s$ for all distinct sets $F,F' \in \mathcal{F}$. 
\begin{itemize}
\item[{\rm (1)}]
 If $s=2d$ for some $d\ge 1$, then
$ |\mathcal{F}| \le \sum_{0\le i\le d} {n \choose i}$. 

\item[{\rm (2)}] 
If $s=2d+1$ for some $d\ge 1$, then
$ |\mathcal{F}| \le \sum_{0\le i\le d} {n \choose i}
+ {n-1 \choose d}$.
\end{itemize}
\end{theorem}

In 2017, Frankl \cite{Fra2017cpc} determined the extremal families 
attaining the upper bound, and he 
also established a stability result on Kleitman's theorem, 
that is, a diameter version of Theorem \ref{thm1.2}.
Inspired by these results, we would like to propose the following problem 
for readers.

\begin{problem}
Is there a diameter version of Theorem \ref{thmmain}? 
\end{problem}

In 2020, Huang, Klurman and Pohoata  \cite{HKP2020} provided an algebraic proof for Theorem \ref{thm-Kleit}, which leads to an algebraic proof of Theorem \ref{thm1.1} as well. 
For the case $s=2d+1$, 
 Gao, Liu and Xu \cite{GLX2022} recently proved a finer stability result  
 on Kleitman's theorem 
  by  an algebraic method. 

\begin{problem}
Are there algebraic proofs of Theorems 
\ref{thm1.2} and \ref{thmmain}? 
\end{problem}

\section*{Acknowledgements}   
We would like to express our gratitude to the anonymous reviewers for their detailed and constructive comments 
which greatly improved the presentation of this paper. 
Yongtao Li was supported by the Postdoctoral Fellowship Program of CPSF (No. GZC20233196). 
Biao Wu was supported by the NSFC (No. 11901193).

\frenchspacing
\bibliographystyle{unsrt}

\begin{thebibliography}{99}

\bibitem{AK1996}
R. Ahlswede,  L.H. Khachatrian, The complete non-trivial intersection theorem
for systems of finite sets, 
J. Comb. Theory, Ser. A, 76 (1996) 121--138.

\bibitem{BL2022}
J. Balogh, W. Linz, 
Short proofs of three results about intersecting systems, Comb. Theory 4 (1) (2024), No. 4.

\bibitem{CLW2021}
M. Cao, B. Lv, K. Wang, The structure of large non-trivial $t$-intersecting families
of finite sets, European J. Combin. 97 (2021), No. 103373.


\bibitem{CLLW2022}
M. Cao, M. Lu, B. Lv, K. Wang, 
Nearly extremal non-trivial cross $t$-intersecting families and $r$-wise $t$-intersecting families, 
European J. Combin. 120 (2024), No. 103958. 


\bibitem{EKR1961}
P. Erd\H{o}s, C. Ko, R. Rado,
Intersection theorems for systems of finite sets,
Quart. J. Math. Oxf. 2(12) (1961) 313--320.

\bibitem{Fra1976}
P. Frankl, The Erd\H{o}s--Ko--Rado theorem
is true for $n=ckt$, in: Combinatorics, vol. I, Proc. Fifth Hungarian Colloq., Keszthey, 1976, in: Colloq. Math. Soc. János Bolyai, vol. 18, North-Holland, 1978, pp. 365--375. 


\bibitem{Fra1978}
P. Frankl, On intersecting families of finite sets, 
J. Comb. Theory, Ser. A, 24 (1978) 146--161.





\bibitem{FT1992}
P. Frankl, N. Tokushige,
Some best possible inequalities concerning cross-intersecting families,
J. Combin. Theory Ser. A 61 (1992) 87--97.



\bibitem{FT2013}
P. Frankl, N. Tokushige,
The Katona theorem for vector spaces,
J. Combin. Theory, Ser. A 120 (2013) 1578--1589.

 \bibitem{Fra}
P. Frankl, New inequalities for cross-intersecting families,
Mosc. J. Comb. Number Theory 6 (4) (2016)  27--32.


\bibitem{Fra2017}
P. Frankl,
A stability result for the Katona theorem,
 J. Combin. Theory Ser. B 122 (2017) 869--876.


\bibitem{Fra2017cpc}
P. Frankl,
A stability result for families with fixed diameter,
Combin. Probab. Comput. 26 (4) (2017) 506--516.

\bibitem{FK2017}
P. Frankl, A. Kupavskii, A size-sensitive inequality for cross-intersecting families,  European J. Combin. 62 (2017) 263--271.



\bibitem{FT2018}
P. Frankl, N. Tokushige,
Extremal Problem for Finite Sets, Student Mathematical Library 86,
Amer. Math. Soc., Providence, RI, 2018.


\bibitem{FW2023}
P. Frankl, J. Wang, 
Intersections and distinct intersections in cross-intersecting families, 
 European J. Combin. 110 (2023), No. 103665.  
 
 \bibitem{FW2024-EUJC}
  P. Frankl, J. Wang, 
 Improved bounds on the maximum diversity of intersecting families, 
European J. Combin. 118 (2024), No. 103885. 

\bibitem{GLX2022}
J. Gao, H. Liu, Z. Xu, 
Stability through non-shadows, 
Combinatorica 43 (6) (2023) 1125--1137. 

 \bibitem{HK2017}
 J. Han, Y. Kohayakawa,
The maximum size of a non-trivial intersecting uniform family
that is not a subfamily of the Hilton--Milner family,
Proc. Amer. Math. Soc. 145 (1) (2017) 73--87.


 \bibitem{HM1967}
A.J.W. Hilton, E.C. Milner,
Some intersection theorems for systems of finite sets,
Q. J. Math. 18 (1967) 369--384.

\bibitem{Hil1976}
A.J.W. Hilton, The Erd\H{o}s--Ko--Rado theorem
with valency conditions, unpublished, 1976.


\bibitem{HKP2020}
H. Huang, O. Klurman, C. Pohoata,
On subsets of the hypercube with prescribed Hamming distances,
J. Combin. Theory  Ser. A 171 (2020), No. 105156. 

\bibitem{HP2022}
Y. Huang, Y. Peng, 
Stability of intersecting families,  
European J. Combin. 115 (2024), No. 103774. 

 \bibitem{Kat1964}
 G.O.H. Katona, Intersection theorems for systems of finite sets, Acta Math. Hungar. 15 (1964) 329--337.

 \bibitem{Kat1966}
 G.O.H. Katona, A theorem on finite sets,
 in: Theory of Graphs, Proceedings,
 Colloq. Tihany, Hungary, (1966) 187--207.



  \bibitem{Kle1966}
 D.J. Kleitman,
 On a combinatorial conjecture of Erd\H{o}s,
 J. Combin. Theory (1966) 209--214.


 \bibitem{KM2017}
A. Kostochka, D. Mubayi, The structure of large intersecting families,
Proc. Amer. Math. Soc. 145 (6)  (2017) 2311--2321.

 \bibitem{Kru1963}
 J.B. Kruskal, The number of simplices in a complex, in: Mathematical Optimization Techniques, Univ. of California Press, Berkeley, 1963, pp. 251--278.



\bibitem{Kup2019}
A. Kupavskii, 
Structure and properties of large intersecting families, 
(2019), arXiv:1710.02440v2.  


\bibitem{MT1989}
M. Matsumoto, N. Tokushige, 
A generalization of the Katona theorem for cross $t$-intersecting families, 
Graphs Combin. 5 (2) (1989) 159--171. 


\bibitem{OV2021}
J. O'Neill, J. Verstra\"{e}te, Non-trivial $d$-wise intersecting families,
J. Combin. Theory Ser. A 178 (2021), No. 105369.

\bibitem{SFQ2022}
C. Shi,  P. Frankl, J. Qian, 
 On non-empty cross-intersecting families, 
 Combinatorica 42 (2022) 1513--1525.

\bibitem{Wil1984}
R.M. Wilson, The exact bound in the Erd\H{o}s--Ko--Rado theorem, Combinatorica 4 (1984) 247--257.

\end{thebibliography}

\end{document}